\documentclass[12pt]{amsart}

\usepackage{amscd,latexsym,amsthm,amsfonts,amssymb,amsmath,amsxtra}
\usepackage[mathscr]{eucal}
\usepackage{hyperref}
\renewcommand\theequation{\thesection.\arabic{equation}}

\newcommand{\BA}{{\mathbb {A}}}

\newcommand{\BC}{{\mathbb {C}}}

\newcommand{\CA}{{\mathcal {A}}}

\newcommand{\CE}{{\mathcal {E}}}
\newcommand{\CF}{{\mathcal {F}}}

\newcommand{\CN}{{\mathcal {N}}}
\newcommand{\CO}{{\mathcal {O}}}

\newcommand{\CS}{{\mathcal {S}}}

\newcommand{\RR}{{\mathrm {R}}}

\newcommand{\RU}{{\mathrm {U}}}

\newcommand{\cusp}{{\mathrm{cusp}}}

\newcommand{\disc}{{\mathrm{disc}}}

\newcommand{\GL}{{\mathrm{GL}}}

\newcommand{\ns}{{\mathrm{ns}}}

\newcommand{\simp}{{\mathrm{sim}}}

\newcommand{\SO}{{\mathrm{SO}}}

\newcommand{\Sp}{{\mathrm{Sp}}}

\newcommand{\st}{{\mathrm{st}}}

\newcommand{\udl}{\underline}
\newcommand{\wt}{\widetilde}

\newcommand{\ul}{\underline}

\newcommand{\bs}{\backslash}

\def\bks{{\backslash}}

\newtheorem{thm}{Theorem}[section]

\newtheorem{lem}[thm]{Lemma}
\newtheorem{prop}[thm]{Proposition}
\newtheorem {conj}[thm]{Conjecture}

\newtheorem {ques/conj}[thm]{Question/Conjecture}

\makeatletter

\newcommand{\Rmnum}[1]{\expandafter\@slowromancap\romannumeral #1@}
\makeatother

\begin{document}
\renewcommand{\theequation}{\arabic{equation}}
\numberwithin{equation}{section}

\title[Fourier Coefficients and Automorphic Forms]{Fourier Coefficients for Automorphic Forms on
Quasisplit Classical Groups}

\author{Dihua Jiang}
\address{School of Mathematics\\
University of Minnesota\\
Minneapolis, MN 55455, USA}
\email{dhjiang@math.umn.edu}

\author{Baiying Liu}
\address{Department of Mathematics\\
University of Utah\\
155 S 1400 E Room 233, Salt Lake City, UT 84112-0090, USA.}
\email{liu@math.utah.edu}


\date{\today}


\thanks{The research of the first named author is supported in part by the NSF Grants DMS--1301567, and the research of the second
named author is supported in part by NSF Grants DMS-1302122, and by  postdoc research funding from Department of Mathematics, University of Utah}

\dedicatory{to James Cogdell, on the occasion of his 60th birthday}

\begin{abstract}
In \cite{J14}, a conjecture was proposed on a relation between the global Arthur parameters
and the structure of Fourier coefficients of the automorphic representations in the corresponding global Arthur packets. In this paper,
we discuss the recent progress on this conjecture and certain problems which lead to better understanding of Fourier coefficients of
automorphic forms. At the end, we extend a useful technical lemma to a few versions, which are more convenient for future applications.
\end{abstract}

\maketitle



\section{Introduction}

Fourier coefficients of automorphic forms are among fundamental objects in the theory of automorphic forms and their applications. In classical
theory of automorphic forms, the algebraic part of Fourier coefficients encodes deep arithmetic information of the automorphic forms
(see \cite{Sa90}, for instance). In the theory of automorphic representations of general quasisplit reductive algebraic groups,
Fourier coefficients are defined to be parameterized by
nilpotent orbits, and are regarded as one of the basic invariants attached to the automorphic representations. For instance,
the Whittaker-Fourier coefficients, which are the Fourier coefficients parameterized to the regular nilpotent orbit, have been playing fundamental
roles in the theory, in particular, in the study of automorphic $L$-functions via the Langlands-Shahidi method and the Rankin-Selberg method.
They are also important for setting up the global converse theorem of Cogdell and Piatetski-Shapiro, which is one of the key ingredients
in the proof of the Langlands functorial transfer from classical groups to general linear groups for generic cuspidal automorphic representations
by Cogdell and his collaborators.

It is a theorem of J. Shalika and of I. Piatetski-Shapiro (\cite{S74} and \cite{PS79})
that every irreducible cuspidal automorphic representation of the
general linear group $\GL_n$ is generic, i.e. has a nonzero Whittaker-Fourier coefficient. On the other hand,
it is also known that when a reductive algebraic group $G$ under consideration is not of $\GL$-type, there
are cuspidal automorphic representations of $G$ that have no nonzero Whittaker-Fourier coefficients. Hence it is natural to consider
Fourier coefficients of different type for automorphic forms of classical groups which are not of $\GL$-type. The power of the
more general Fourier coefficients has been shown in the theory of automorphic descents by Ginzburg, Rallis and Soudry in \cite{GRS11},
and in the construction of endoscopy correspondences for classical groups in \cite{J14}.
Those important constructions use only some special types of the general Fourier coefficients of automorphic representations.

In this paper, we discuss a conjecture made in \cite[Section 4]{J14} that relates the notion of general Fourier coefficients to
the global Arthur parameters for automorphic representations occurring in the discrete spectrum of classical groups.
We first recall briefly from \cite{Ar13} and \cite{Mk14} the global Arthur parameters and decomposition of the discrete spectrum
of an $F$-quasisplit classical group $G$ where $F$ is a number field.  Then we summarize the progress
toward the conjecture and relevant problems on understanding of Fourier coefficients of automorphic forms in general, based on our recent
work in a series of papers (\cite{JL13a}, \cite{JL13b}, \cite{JL14}, \cite{JL15}, and \cite{L13}),
including a joint paper with G. Savin (\cite{JLS14}).
By the end, we extend a useful technical lemma to a few versions, which play important roles in our understanding of
Fourier coefficients for automorphic forms in general. The preliminary version of such lemmas played an important role in the
theory of automorphic descents in \cite{GRS11}, in particular, in the proof of non-vanishing of the automorphic descents.
The extended versions that are discussed in this paper have been and will continue to
be important to the understanding of the global non-vanishing of Fourier coefficients and other related problems.
We refer to \cite{JLS14} for discussions of the ideas related to such technical lemmas from more representation-theoretic point of view.

We thank James Cogdell very much for his encouragement and support over the years and hope that this paper is appropriate for this special volume in honor of him. We also thank the referee for many helpful comments and suggestions.

\section{Arthur Parameters and the Discrete Spectrum}

Following \cite[Chapter 1]{Ar13} and \cite[Chapter 2]{Mk14}, we discuss briefly the global Arthur parameters for quasisplit classical groups and
the endoscopy classification of the discrete spectrum. We recall from \cite[Section 2]{J14} for what will be needed for this current paper.

Let $F$ be a number field and $E$ be a quadratic extension of $F$, whose Galois group is denoted by $\Gamma_{E/F}=\{1,c\}$. The $F$-quasisplit
classical groups $G_n$ are of the following types.
They are either the $F$-quasisplit unitary groups
$\RU_{2n}=\RU_{E/F}(2n)$ or $\RU_{2n+1}=\RU_{E/F}(2n+1)$ of hermitian type, the $F$-split special orthogonal group $\SO_{2n+1}$, the symplectic group
$\Sp_{2n}$, or the $F$-quasisplit even special orthogonal group $\SO_{2n}$.
Let $F'$ be a number field, which is $F$ if $G_n$ is not a unitary group, and
is $E$ if $G_n$ is a unitary group. Denote by $\RR_{F'/F}(n):=\RR_{F'/F}(\GL_{n})$ the Weil restriction of the $\GL_{n}$ from $F'$ to $F$.

Following \cite{Ar13} and \cite{Mk14}, take the closed subgroup $G(\BA)^1$ of $G(\BA)$ given by
\begin{equation}
G(\BA)^1:=\{x\in G(\BA)\mid |\chi(x)|_\BA=1, \forall \chi\in X(G)_F\}
\end{equation}
where $X(G)_F$ is the group of all $F$-rational characters of $G$. It follows that $G(F)$ is a discrete subgroup of $G(\BA)^1$ and
its quotient $G(F)\bks G(\BA)^1$ has finite volume with respect to the Tamagawa measure. We consider
the space of all square integrable functions on $G(F)\bks G(\BA)^1$, which is denoted by
$L^2(G(F)\bks G(\BA)^1)$. It has the following embedded, right $G(\BA)^1$-invariant Hilbert subspaces
\begin{equation}
L^2_{\cusp}(G(F)\bks G(\BA)^1)\subset L^2_{\disc}(G(F)\bks G(\BA)^1)\subset L^2(G(F)\bks G(\BA)^1)
\end{equation}
where $L^2_{\disc}(G(F)\bks G(\BA)^1)$ is called the discrete spectrum of $G(\BA)^1$, which decomposes into
a direct sum of irreducible representations of $G(\BA)^1$; and $L^2_{\cusp}(G(F)\bks G(\BA)^1)$ is called the
cuspidal spectrum of $G(\BA)^1$, which consists of all cuspidal automorphic functions on $G(\BA)^1$ and is
embedded into $L^2_{\disc}(G(F)\bks G(\BA)^1)$.

As in \cite{J14}, we denote by $\CA_\cusp(G)$
the set of irreducible unitary representations of $G(\BA)$ whose restriction to $G(\BA)^1$ are irreducible constituents of
$L^2_{\cusp}(G(F)\bks G(\BA)^1)$. Similar definitions apply to $\CA_2(G)$ and $\CA(G)$. It is clear that
\begin{equation}
 \CA_\cusp(G)\subset \CA_2(G)\subset \CA(G).
\end{equation}

Following \cite{Ar13} and \cite{Mk14} (see also \cite{J14}), the set of all $F$-quasisplit simple classical groups,
which are the twisted endoscopic groups of $\RR_{F'/F}(\GL_N)$, is denoted by $\wt{\CE}_\simp(N)$.
Note that when $F'=E$, the endoscopic data (in particular, the endoscopic group $G_\psi$) contains the $L$-embedding $\xi_{\chi_\kappa}$ from the
$L$-group $^LG_\psi$ to the $L$-group of $\RR_{E/F}(\GL_N)$, which depends on the sign $\kappa=\pm$. The two embeddings are not equivalent.

For each $G\in\wt{\CE}_\simp(N)$, $\wt{\Psi}_2(G)$ denotes the global Arthur parameters for the discrete spectrum of $G$.
Theorem 1.5.2 of \cite{Ar13} and Theorem 2.5.2 of \cite{Mk14} give
the following decomposition of the discrete spectrum of $G$.
\begin{equation}\label{arthur}
L^2_\disc(G(F)\bks G(\BA))
\cong
\oplus_{\psi\in\wt{\Psi}_2(G)}\oplus_{\pi\in\wt{\Pi}_\psi(\epsilon_\psi)}m_\psi\pi,
\end{equation}
where $\epsilon_\psi$ is a linear character defined as in Theorem 1.5.2 of \cite{Ar13} in terms of $\psi$ canonically, and
$\wt{\Pi}_\psi(\epsilon_\psi)$ is the subset of the global Arthur packet $\wt{\Pi}_\psi$ consisting of
all elements $\pi=\otimes_v\pi_v$ whose characters are equal to $\epsilon_\psi$, and finally, $m_\psi$ is the multiplicity
depending only on the global Arthur parameter $\psi$. See Theorem 1.5.2 of \cite{Ar13}
and Theorem 2.5.2 of \cite{Mk14} for details.

It should be mentioned that in the case of unitary groups, the set $\wt{\Psi}_2(G)$ depends on
a given embedding $\xi_{\chi_\kappa}$ of $L$-groups, as described in \cite[Section 2.1]{Mk14}. Hence one should use the notation
$\wt{\Psi}_2(G,\xi_{\chi_\kappa})$ for $\wt{\Psi}_2(G)$. If $\xi_{\chi_\kappa}$ is fixed in a discussion, it may be dropped from the
notation if there is no confusion. As explained in \cite{J14}, for simplicity, the set $\wt{\Pi}_\psi(\epsilon_\psi)$ is called the
{\sl automorphic $L^2$-packet} attached to the global Arthur parameter $\psi$. In the following section, we describe a conjecture,
which was first given in \cite{J14}, on the understanding of the structure of Fourier coefficients of members in
the automorphic $L^2$-packet $\wt{\Pi}_\psi(\epsilon_\psi)$ in terms of the global Arthur parameter $\psi$.

\section{A Conjecture on Fourier Coefficients}

In order to state the conjecture in \cite[Section 4]{J14} on the structure of Fourier coefficients for members in
the automorphic $L^2$-packet $\wt{\Pi}_\psi(\epsilon_\psi)$, we refer to \cite[Section 4]{J14} for the details of the
notation and basic facts needed here.

Let $G\in\wt{\CE}_\simp(N)$ be an $F$-quasisplit classical group and $\frak{g}$ be the Lie algebra of the algebraic group $G$.
Let $\CN(\frak{g})$ be the nilcone of $\frak{g}$, which is an algebraic variety defined over $F$ and stable under the
adjoint action of $G$. We fix an algebraic closure $\overline{F}$ of $F$.
The $\overline{F}$-points $\CN(\frak{g})(\overline{F})$ decomposes into
finitely many adjoint $G(\overline{F})$-orbits $\CO$. It is a well-known theorem (Chapter 5 of \cite{CM93}) that
the set of those finitely many nilpotent orbits is in one-to-one correspondence with the set of partitions of $N_G$
with certain parity constraints (called $G$-partitions of $N_G$ or partitions of type $(N_G,G)$).
Here $N_G=N+1$ if $G$ is of type $\SO_{2n+1}$,
$N_G=N-1$ if $G$ is of type $\Sp_{2n}$, and $N_G=N$ if $G$ is of type $\SO_{2n}$ or a unitary group $\RU_N=\RU(N)$.

Over $F$, the $F$-rational points $\CN(\frak{g})(F)$ decomposes into $F$-stable adjoint $G(F)$-orbits $\CO^\st$, which are parameterized by
the corresponding $G$-partitions of $N_G$. Let $\udl{p}$ be a partition of type $(N_G,G)$, and denote the corresponding $F$-stable orbit
by $\CO^\st_{\udl{p}}$. An explicit description of the $F$-rational orbits within a given $F$-stable orbit $\CO^\st_{\udl{p}}$ is
given in \cite[Section I.6]{W01}.

Following \cite[Section 4]{J14}, for $X\in\CN(\frak{g})(F)$, via the theory of the
$\frak{sl}_2$-triple (over $F$), one defines the unipotent $F$-subgroup of $G$, which is denoted by $V_X$, and a
nontrivial character $\psi_X$ on $V_X(\BA)$, which is trivial on $V_X(F)$, from a nontrivial character $\psi_F$ of
$F\bks\BA$. Let $\pi$ belong to $\CA(G)$. The {\bf Fourier coefficient} of $\varphi_\pi$ in the space of $\pi$ is defined by
$$
\CF^{\psi_X}(\varphi_\pi)(g):=\int_{V_X(F)\bks V_X(\BA)}\varphi_\pi(vg)\psi_X^{-1}(v)dv.
$$
Since $\varphi_\pi$ is automorphic, the nonvanishing of $\CF^{\psi_X}(\varphi_\pi)$ depends only on the $G(F)$-adjoint orbit $\CO_X$ of
$X$. Hence we may denote the Fourier coefficient of $\varphi_\pi$ by $\CF^{\psi_{\udl{p},\CO}}(\varphi_\pi)$, with
$X\in\CO$ and $\CO\subset\CO^\st_{\udl{p}}$.

We say that a $\pi\in\CA(G)$ has a nonzero $\psi_{\udl{p},\CO}$-Fourier coefficient if there exists an
automorphic form $\varphi$ in the space of $\pi$ with a nonzero $\psi_{\udl{p},\CO}$-Fourier coefficient
$\CF^{\psi_{\udl{p},\CO}}(\varphi_\pi)$ for some $F$-rational orbit $\CO$ in the $F$-stable orbit $\CO^\st_{\udl{p}}$.
For any $\pi\in\CA(G)$, as in \cite{J14}, we define $\frak{p}^m(\pi)$ (which corresponds to $\mathfrak{n}^m(\pi)$ in the notation of \cite{J14})
to be the set of all partitions $\underline{p}$ of type $(N_G,G)$ with the following properties:
\begin{enumerate}
\item $\pi$ has a nonzero $\psi_{\udl{p}, \CO}$-Fourier coefficient
for some $F$-rational orbit $\CO$ in the $F$-stable orbit $\CO^\st_{\udl{p}}$, but
\item for any ${\udl{p}'} > \udl{p}$ (with the natural ordering of partitions), $\pi$ has no nonzero Fourier coefficients
attached to ${\udl{p}'}$.
\end{enumerate}
It is clear that the local analogy of the set $\frak{p}^m(\pi)$ can be defined for irreducible admissible representations of reductive
algebraic groups over $p$-adic local fields. We refer the local theory to \cite{Ka87} and \cite{MW87} for details.
In \cite[Section4]{J14}, one may also find discussions and references for the relevant local theory for complex and real cases.

\begin{conj}[\cite{Ka87}, \cite{MW87} and \cite{GRS03}]\label{singleton}
With notation as above, the $\frak{p}^m(\pi)$ contains only one partition for local case and automorphic case.
\end{conj}

As explained in \cite[Section 4]{J14}, for each global Arthur parameter $\psi\in\wt{\Phi}_2(G)$  that is of the form
\begin{equation}\label{ap}
\psi=(\tau_1,b_1)\boxplus(\tau_2,b_2)\boxplus\cdots\boxplus(\tau_r,b_r)
\end{equation}
where $\tau_i\in\CA_\cusp(\GL_{a_i})$,
one defines $\underline{p}(\psi)=[(b_1)^{(a_1)}\cdots(b_r)^{(a_r)}]$ to be the partition of type $(N_G,G^\vee(\BC))$ attached to the
global Arthur parameter $\psi$.
For $\pi\in\wt{\Pi}_\psi(\epsilon_\psi)$, the structure of the global Arthur parameter $\psi$ is expected to impose constraints on
the structure of $\frak{p}^m(\pi)$, as described in the following conjecture.
We denote by $\eta_{{\frak{g}^\vee,\frak{g}}}$ the Barbasch-Vogan duality map
$($see \cite[Definition A1]{BV85} and \cite[Section 3.5]{Ac03}$)$ from the partitions for $G^\vee$ to
the partitions for $G$.

\begin{conj}[Conjecture 4.2, \cite{J14}]\label{cubmfc}
Let $G\in\wt{\CE}_\simp(N)$ be an $F$-quasisplit classical group. For any global Arthur parameter $\psi\in\wt{\Psi}_2(G)$,
let $\wt{\Pi}_{\psi}(\epsilon_\psi)$ be the automorphic $L^2$-packet attached to $\psi$. Then the following hold:
\begin{enumerate}
\item[(1)] Any partition $\ul{p}$ of type $(N_G,G^\vee(\BC))$, if $\ul{p}>\eta_{{\frak{g}^\vee,\frak{g}}}(\underline{p}(\psi))$, does
not belong to $\frak{p}^m(\pi)$ for any $\pi\in\wt{\Pi}_{\psi}(\epsilon_\psi)$.
\item[(2)] For every $\pi\in\wt{\Pi}_{\psi}(\epsilon_\psi)$, any partition $\ul{p}\in\frak{p}^m(\pi)$ has the property that
$\ul{p}\leq \eta_{{\frak{g}^\vee,\frak{g}}}(\ul{p}(\psi))$.
\item[(3)] There exists at least one member $\pi\in\wt{\Pi}_{\psi}(\epsilon_\psi)$ having the property that
$\eta_{{\frak{g}^\vee,\frak{g}}}(\ul{p}(\psi))\in \frak{p}^m(\pi)$.
\end{enumerate}
\end{conj}

It is not hard to see that Part (2) is stronger than Part (1) in Conjecture \ref{cubmfc}. However, one expects to prove Part (2) from Parts
(1) and (3) with help of Conjecture \ref{singleton}. More discussions will be found in Section 4.

It is clear that Conjecture \ref{cubmfc} is essentially about the partitions to which a nonzero Fourier coefficient is attached, for
$\pi\in\wt{\Pi}_{\psi}(\epsilon_\psi)$, and yields no information about the rational structure of the $F$-orbits within the
$F$-stable orbit $\CO^\st_{\udl{p}}$. At the partition level,
it remains interesting and important to figure out exactly what $\frak{p}^m(\pi)$ is for each individual $\pi\in\wt{\Pi}_{\psi}(\epsilon_\psi)$.
In general, Conjecture \ref{singleton} asserts that $\frak{p}^m(\pi)$ only contains one partition! Furthermore, the determination of
the rational structure of the $F$-nilpotent orbits with the $F$-stable orbits $\CO^\st_{\udl{p}}$ with $\udl{p}\in\frak{p}^m(\pi)$
will be a more interesting and harder problem. We refer to Section 5 for more discussions on these problems.

When the global Arthur parameter $\psi\in\wt{\Psi}_2(G)$ is generic, it is not hard to show that Conjecture \ref{cubmfc} holds.
As suggested by the referee, we state this as a theorem.

\begin{thm}\label{gap}
Let $G$ be an $F$-quasisplit classical group.
For any generic global Arthur parameter $\psi\in\wt{\Psi}_2(G)$, Conjecture \ref{cubmfc} holds for the automorphic $L^2$-packet
$\wt{\Pi}_{\psi}(\epsilon_\psi)$ attached to $\psi$.
\end{thm}

\begin{proof}
We recall that a global Arthur parameter $\psi$ as in \eqref{ap} is called {\sl generic}, as defined in \cite[Chapter 1]{Ar13}, if
$b_i=1$ for $i=1,2,\cdots,r$. The partition $\udl{p}(\psi)$ attached to any generic global Arthur parameter $\psi$ is the trivial
partition $[1^*]$. It is an easy calculation that the Barbasch-Vogan duality $\eta_{{\frak{g}^\vee,\frak{g}}}(\underline{p}(\psi))$ must be
the principal partition of type $(N_G,G)$. Hence Parts (1) and (2) are automatic. It remains to prove Part (3).

According to \cite[Theorem 1.5.2]{Ar13} and \cite[Theorem 2.5.2]{Mk14}, any member $\pi$ in the automorphic $L^2$-packet
$\wt{\Pi}_{\psi}(\epsilon_\psi)$ attached to $\psi$ has the isobaric representation
$$
\tau:=\tau_1\boxplus\tau_2\boxplus\cdots\boxplus\tau_r,
$$
following the notation in \eqref{ap}, as the image of the global Arthur-Langlands transfer to $\GL_N$. Note that as part of the definition
of global Arthur parameters, $\tau_i\not\cong\tau_j$ if $i\neq j$.
Now by \cite{GRS11}, the automorphic descent method constructs an irreducible generic cuspidal automorphic representation
$\pi_0$ of $G(\BA)$ which
has the Langlands functorial transfer $\tau$. Hence $\pi_0$ belongs to $\wt{\Pi}_{\psi}(\epsilon_\psi)$ and has a nonzero
Fourier coefficient parameterized by $\eta_{{\frak{g}^\vee,\frak{g}}}(\underline{p}(\psi))$, which is the principal partition of type $(N_G,G)$.
This proves Part (3).
\end{proof}

By the local endoscopic classification of tempered representations (\cite{Ar13} and \cite{Mk14}) any global tempered $L$-packet has
a generic Arthur parameter and hence, by Theorem \ref{gap}, contains a generic member. This confirms the global version of the Shahidi conjecture
that any global tempered $L$-packet has a member with a non-zero Whittaker-Fourier coefficient.

On the other hand, we consider a global Arthur parameter of simple type,
$\psi=(\tau,b)\in\wt{\Psi}_2(G)$,
with $\tau\in\CA_\cusp(a)$ being conjugate self-dual.

If take $G=\SO_{2n+1}$, then $N_G=2n=ab$. Hence $\tau$ is of orthogonal type if and only if $b=2l$; and
$\tau$ is of symplectic type if and only if $b=2l+1$. The partition attached to $\psi$ is
$\udl{p}(\psi)=[b^a]$, and the Barbasch-Vogan dual is
$$
\eta_{{\frak{g}^\vee,\frak{g}}}(\underline{p}(\psi))=
\begin{cases}
[(a+1)a^{b-2}(a-1)1]&\text{if}\ b=2l \text{ and a is even};\\
[a^b1]&\text{if}\ b=2l \text{ and a is odd};\\
[(a+1)a^{b-1}]&\text{if}\ b=2l+1.
\end{cases}
$$

If take $G=\SO_{2n}$, then $N_G=2n=ab$. Hence $\tau$ is of orthogonal type if and only if $b=2l+1$; and
$\tau$ is of symplectic type if and only if $b=2l$. The partition attached to $\psi$ is
$\udl{p}(\psi)=[b^a]$, and the Barbasch-Vogan dual is
$$
\eta_{{\frak{g}^\vee,\frak{g}}}(\underline{p}(\psi))=
\begin{cases}
[a^b]&\text{if}\ b=2l;\\
[a^{b-1}(a-1)1]&\text{if}\ b=2l+1.
\end{cases}
$$

If take $G=\Sp_{2n}$, then $N_G=2n+1=ab$. Hence $\tau$ must be of orthogonal type and $b=2l+1$ and $a=2e+1$. The partition attached to $\psi$ is
$\udl{p}(\psi)=[b^a]$, and the Barbasch-Vogan dual is
$$
\eta_{{\frak{g}^\vee,\frak{g}}}(\underline{p}(\psi))=[a^{b-1}(a-1)].
$$

It seems not so straightforward to deal with Conjecture \ref{cubmfc} for simple global Arthur parameters $\psi=(\tau,b)$, and hence it is expected
to be more involved in the general case.

We will discuss the recent progress towards Conjecture \ref{cubmfc}. For example, when $G=\Sp_{2n}$, take Arthur parameters of form $\psi=(\tau,b)\boxplus (1_{\GL_1(\BA)}, 1)$, where $\tau$ is an irreducible cuspidal representation of
$\GL_{2k}(\BA)$ and is of symplectic type, and $b$ is even, one has that $\ul{p}(\psi)=[b^{(2k)}1]$. In this case,
Conjecture \ref{cubmfc} has been proved by the second named author in \cite{L13}, where it is also shown that
$\frak{p}^m(\pi)$ contains only one partition for certain residual representations in the automorphic $L^2$-packets corresponding to this particular family of global Arthur parameters.

\section{Progress towards Conjecture \ref{cubmfc}}

We discuss what have been proved about Conjecture \ref{cubmfc} when $G=\Sp_{2n}$. First of all, the work of Ginzburg-Rallis-Soudry
(\cite{GRS03}) provides with preliminary results on Fourier coefficients of cuspidal automorphic representations of $\Sp_{2n}(\BA)$.
We continue with the topic in order to understand Conjecture \ref{cubmfc}.
Here we summarize what we have done and discuss what more needs to be done. Hopefully, the methods are extendable to
all quasisplit classical groups.

\subsection{Part (1) of Conjecture \ref{cubmfc}}
For a general global Arthur parameter $\psi\in\wt{\Psi}_2(\Sp_{2n})$,
Part (1) of Conjecture \ref{cubmfc} is completely proved in \cite{JL13b}, using local-global argument.
All representations $\pi$ in $\wt{\Pi}_{\psi}(\epsilon_\psi)$ are nearly equivalent, i.e.,
for almost all finite places, they have the same unramified local component $\pi_v$ which corresponds to the local Arthur parameter $\psi_v$.
On one hand, any Fourier coefficient associated to $\ul{p}$ gives a local functional for $\pi_v$,
and produces the corresponding Jacquet module. On the other hand, using the theory of local descent, one can show that
for any symplectic partition $\ul{p} >\eta_{{\frak{g}^\vee,\frak{g}}}(\underline{p}(\psi))$,
$\pi_v$ has no nonzero such Jacquet modules corresponding to Fourier coefficients associated to $\ul{p}$. We expect that the same
arguments applicable to other quasisplit classical groups.

\subsection{Part (2) of Conjecture \ref{cubmfc}}
If we assume Conjecture \ref{singleton} that $\frak{p}^m(\pi)$ contains only one partition, then Part (2) of Conjecture \ref{cubmfc}
essentially follows from Parts (1) and (3) of Conjecture \ref{cubmfc} plus certain local constraints at unramified local places as
discussed in \cite{JL13b}. However, without assuming Conjecture \ref{singleton},
Part (2) of Conjecture \ref{cubmfc} is also settled in \cite{JL13b} partially, namely, any symplectic partition $\ul{p}$ of $2n$, if $\ul{p}$ is bigger than $\eta_{{\frak{g}^\vee,\frak{g}}}(\underline{p}(\psi))$ under the lexicographical ordering, does
not belong to $\frak{p}^m(\pi)$ for any $\pi\in\wt{\Pi}_{\psi}(\epsilon_\psi)$. It seems not an easy problem in general to resolve
Part (2) of Conjecture \ref{cubmfc} without Conjecture \ref{singleton}, which by itself is a hard problem.

\subsection{Part (3) of Conjecture \ref{cubmfc}: residual case}
We start in \cite{JL14} to investigate Part (3) of Conjecture \ref{cubmfc} for $G=\Sp_{2n}$.

The idea is to construct or determine a particular
member in a given automorphic $L^2$-packet $\wt{\Pi}_\psi(\epsilon_\psi)$ attached to a general global Arthur parameter $\psi$, whose
Fourier coefficients achieve the partition $\eta_{{\frak{g}^\vee,\frak{g}}}(\ul{p}(\psi))$. We expect that such members should be the
distinguished members in $\wt{\Pi}_\psi(\epsilon_\psi)$, following the Whittaker normalization in the sense of Arthur for global generic
Arthur parameters (\cite{Ar13}). Since Theorem \ref{gap} confirms Conjecture \ref{cubmfc} for all generic global Arthur parameters,
we will consider only the global Arthur parameters which are non-generic.

For general non-generic global Arthur parameters, the distinguished members in $\wt{\Pi}_\psi(\epsilon_\psi)$ can be
certain residual representations determined by $\psi$ as conjectured by M\oe glin in \cite{M08} and \cite{M11}, or
certain cuspidal automorphic representations, which may be explicitly constructed through the framework of endoscopy
correspondences as outlined in \cite{J14}. Due to the different nature of the construction methods, we have to treat them separately,
in order to prove Part (3) of Conjecture \ref{cubmfc}. Of course, it will also be very interesting to find a proof of
Part (3) of Conjecture \ref{cubmfc} without using construction methods. For the moment, the authors have no idea to proceed with this approach.

When the distinguished members $\pi$ in a given $\wt{\Pi}_\psi(\epsilon_\psi)$ are residual representations, our method is to establish,
through the explicit construction of residual representations from the given cuspidal data,
the nonvanishing of a Fourier coefficients of $\pi$ attached to the partition $\eta_{{\frak{g}^\vee,\frak{g}}}(\ul{p}(\psi))$, in terms of
the construction data, i.e. the non-generic global Arthur parameter $\psi$ that we started with.

It is clear that such a method can be viewed as a natural extension of the well-known
Langlands-Shahidi method from generic Eisenstein series (\cite{Sh10}) to non-generic Eisenstein series, and in  particular to the singularity
of Eisenstein series, i.e. the residues of Eisenstein series. Of course, this method can also be viewed as an extension of the
automorphic descent method of Ginzburg-Rallis-Soudry for a particular residual representations (\cite{GRS11}) to general residual representations.
For more detailed discussions in this aspect, we refer to \cite{J14}.

In \cite{JL14}, we tested our method for these non-generic global Arthur parameters $\psi$, whose
automorphic $L^2$-packets $\wt{\Pi}_\psi(\epsilon_\psi)$ contain the residual representations that are completely determined
in our work joint with Zhang (\cite{JLZ13}). Those non-generic global Arthur parameters of $\Sp_{2n}(\BA)$ are of the following form
$$
\psi=(\tau_1,b_1) \boxplus \boxplus_{i=2}^r (\tau_i,1),\ \text{with}\  b_1 > 1.
$$
This leads to the following three cases, depending on the symmetry of $\tau_1$ and the relation of $\tau_1$ with $\tau_i$ for $i=2,3,\cdots,r$.

\textbf{Case \Rmnum{1}:} $\psi=(\tau, 2b+1) \boxplus \boxplus_{i=2}^r (\tau_i,1),$
where $b \geq 1$, $\tau \ncong \tau_i$, for any $2 \leq i \leq r$.

\textbf{Case \Rmnum{2}:} $\psi=(\tau, 2b+1) \boxplus (\tau, 1) \boxplus_{i=3}^r (\tau_i,1),$
where $b \geq 1$, $\tau \ncong \tau_i$, for any $3 \leq i \leq r$.

\textbf{Case \Rmnum{3}:} $\psi=(\tau, 2b) \boxplus \boxplus_{i=2}^r (\tau_i,1),$
where $b \geq 1$.

Since the global Arthur parameter $\psi$ is of orthogonal type, the irreducible cuspidal automorphic representation $\tau$ of $\GL_a(\BA)$
is of orthogonal type in \textbf{Case \Rmnum{1}} and \textbf{Case \Rmnum{2}}, and of symplectic type in \textbf{Case \Rmnum{3}}. Of course,
the $\tau_i$'s are of orthogonal type in all three cases.

When $\tau$ is of orthogonal type and hence in both \textbf{Case \Rmnum{1}} and \textbf{Case \Rmnum{2}}, the corresponding
residual representations described in \cite{JLZ13} are always nonzero.
When $\tau$ is of symplectic type, as long as $r\geq 2$, the relation between
$\tau$ and $\tau_i$ determined by the non-vanishing of the central values $L(\frac{1}{2},\tau\times\tau_i)$, for $i=2,3,\cdots,r$,
controls the non-vanishing of the corresponding residual representation in this case. Hence the corresponding Gan-Gross-Prasad conjecture (\cite{GGP12}) is expected to be involved in the structure of the automorphic $L^2$-packet $\wt{\Pi}_{\psi}(\epsilon_\psi)$ and
the structure of Fourier coefficients for the members therein. We investigated the case when
$\wt{\Pi}_{\psi}(\epsilon_\psi)$ contains a residual representation. When the automorphic $L^2$-packet $\wt{\Pi}_{\psi}(\epsilon_\psi)$
does not contain any residual representation, the situation is more involved, and will be discussed briefly in the next section.

Let $G=\Sp_{2n}$ and let $P^{2n}_r = M^{2n}_r N^{2n}_r$ (with $1 \leq r \leq n$) be the standard parabolic subgroup
of ${\Sp}_{2n}$ with Levi part $M^{2n}_r$ isomorphic to $\GL_r \times \Sp_{2n-2r}$,
$N^{2n}_r$ is the unipotent radical. The discussion may also involved in the metaplectic double cover $\wt{\Sp}_{2n}(\BA)$ of $\Sp_{2n}(\BA)$
when consider the Fourier-Jacobi coefficients. In such a situation, we let
$\wt{P}^{2n}_r(\BA)=\wt{M}^{2n}_r(\BA) N^{2n}_r(\BA)$ be the pre-image of $P^{2n}_r(\BA)$ in $\wt{\Sp}_{2n}(\BA)$ (the superscript $2n$ may be dropped when there is no confusion).

{\bf Case I:}\
 $\psi$ has the following form
 \begin{equation}\label{case1psi}
 \psi=(\tau, 2b+1) \boxplus \boxplus_{i=2}^r (\tau_i,1),
 \end{equation}
where $b \geq 1$, $\tau \ncong \tau_i$, for any $2 \leq i \leq r$. Assume $\tau$ is an irreducible cuspidal representation of $\GL_a(\BA)$ with central character $\omega_{\tau}$, $\tau_i$ is an irreducible cuspidal representation of $\GL_{a_i}(\BA)$ with central character $\omega_{\tau_i}$, $2 \leq i \leq r$.
Then, by the definition of the global Arthur parameters for $\Sp_{2n}(\BA)$,
$\tau$ and $\tau_i$'s are all of orthogonal type; $\tau_i \ncong \tau_j$, for $2 \leq i \neq j \leq r$; $2n+1=a(2b+1)+\sum_{i=2}^r {a_i}$, and $\omega_{\tau}^{2b+1} \cdot \prod_{i=2}^r \omega_{\tau_i} = 1$. Consider the isobaric sum representation
$$\pi=\tau \boxplus \tau_2 \boxplus \cdots \boxplus \tau_r$$
of $\GL_{2m+1}(\BA)$, where $2m+1=a+\sum_{i=2}^r a_i = 2n+1-2ab$. Note that $\pi$ has central character $\omega_{\pi}=\omega_{\tau} \cdot \prod_{i=2}^r \omega_{\tau_i} = 1$ and the following
$$
a\leq 2m+1= 2n+1-2ab
$$
holds.

By the theory of automorphic descent \cite[Theorem 3.1]{GRS11}, there is an irreducible generic cuspidal representation $\sigma$ of
$\Sp_{2n-2ab}(\BA)$ such that $\sigma$ has the functorial transfer $\pi$. This can be regarded as Part (3) of
Conjecture \ref{cubmfc} for the generic global Arthur parameter
$$
\psi_\pi=(\tau,1) \boxplus (\tau_2,1) \boxplus \cdots \boxplus (\tau_r,1).
$$
It follows that $L(s, \tau \times \sigma)$ has a (simple) pole at $s=1$.

Let $\Delta(\tau, b)$ be the Speh residual representation in the discrete spectrum of
$\GL_{ab}(\BA)$ (see \cite{MW89}, or \cite[Section 1.2]{JLZ13}).
For any automorphic form
$$
\phi \in \CA(N_{ab}(\BA)M_{ab}(F) \bs \Sp_{2ab+2m}(\BA))_{\Delta(\tau,b) \otimes \sigma},
$$
following \cite{L76} and \cite{MW95}, one has a residual Eisenstein series
$$
E(\phi,s)(g)=E(g,\phi_{\Delta(\tau,b) \otimes \sigma},s).
$$
We refer \cite{JLZ13} for particular details about this family of Eisenstein series. In particular, it is proved in \cite{JLZ13} that
$E(\phi,s)(g)$ has a simple pole at $\frac{b+1}{2}$, which is the right-most one. We denote by $\CE(g,\phi)$ the residue,
which is square-integrable. They generate the residual representation $\CE_{\Delta(\tau, b)\otimes \sigma}$ of $\Sp_{2n}(\BA)$.
Following \cite[Section 6.2]{JLZ13}, the global Arthur parameter of this nonzero square-integrable automorphic representation
$\CE_{\Delta(\tau,b)\otimes \sigma}$ is exactly
$$
\psi=(\tau, 2b+1) \boxplus \boxplus_{i=2}^r (\tau_i,1)
$$
as in \eqref{case1psi}.

\begin{thm}[\cite{JL14}]\label{main1}
For any global Arthur parameter of the form
$$
\psi=(\tau, 2b+1) \boxplus \boxplus_{i=2}^r (\tau_i,1)
$$
with $b \geq 1$, and $\tau \ncong \tau_i$ for any $2 \leq i \leq r$,
the residual representation $\CE_{\Delta(\tau, b)\otimes \sigma}$ has a nonzero Fourier coefficient attached to
the Barbasch-Vogan duality $\eta_{\frak{so}_{2n+1}, \frak{sp}_{2n}}(\ul{p}(\psi))$ of the partition $\ul{p}(\psi)$ associated to
$(\psi,\SO_{2n+1}(\BC))$.
\end{thm}

Before we go into the idea of the proof we explicate the partition
$\eta_{\frak{so}_{2n+1}, \frak{sp}_{2n}}(\ul{p}(\psi))$.
In general,
given any partition $\ul{q} =[q_1 q_2 \cdots q_r]$ for $\frak{so}_{2n+1}(\BC)$ with $q_1 \geq q_2 \geq \cdots \geq q_r > 0$,
whose even parts occurring with even multiplicity. Let $\ul{q}^- =[q_1 q_2 \cdots q_{r-1} (q_r-1)]$, the Barbasch-Vogan duality
$\eta_{\frak{so}_{2n+1}, \frak{sp}_{2n}}$, following \cite[Definition A1]{BV85} and \cite[Section 3.5]{Ac03}, is defined by
$$
\eta_{\frak{so}_{2n+1}, \frak{sp}_{2n}}
(\ul{q}) := ((\ul{q}^-)_{\Sp_{2n}})^t,
$$
where $(\ul{q}^-)_{\Sp_{2n}}$ is the $\Sp_{2n}$-collapse of $\ul{q}^-$, which is the biggest special symplectic partition which is smaller than $\ul{q}^-$.

For $\psi=(\tau, 2b+1) \boxplus \boxplus_{i=2}^r (\tau_i,1)$ in \textbf{Case \Rmnum{1}}, as discussed in \cite[Section 4]{J14}, one has
$$
\ul{p}(\psi)=[(2b+1)^{a}(1)^{2m+1-a}].
$$
When $a=2m+1$, one obtains that
\begin{align*}
\eta_{\frak{so}_{2n+1}, \frak{sp}_{2n}}(\ul{p}(\psi))
= \ & \eta_{\frak{so}_{2n+1}, \frak{sp}_{2n}}([(2b+1)^{a}(1)^{2m+1-a}])\\
= \ & (([(2b+1)^{a}]^-)_{\Sp_{2n}})^t \\
= \ & [(2b+1)^{a-1}(2b)]^t\\
= \ & [(a-1)^{2b+1}]+[(1)^{2b}]\\
= \ & [(a)^{2b}(a-1)]\\
= \ & [(a)^{2b}(2m)].
\end{align*}
When $a \leq 2m$ and $a$ is even, one obtains that
\begin{align*}
\eta_{\frak{so}_{2n+1}, \frak{sp}_{2n}}(\ul{p}(\psi))
= \ & \eta_{\frak{so}_{2n+1}, \frak{sp}_{2n}}([(2b+1)^{a}(1)^{2m+1-a}])\\
= \ & [(2b+1)^a(1)^{2m-a}]^t\\
= \ & [(a)^{2b+1}] + [(2m-a)]\\
= \ & [(2m)(a)^{2b}].
\end{align*}
Finally, when $a \leq 2m$ and $a$ is odd, one obtains that
\begin{align*}
\eta_{\frak{so}_{2n+1}, \frak{sp}_{2n}}(\ul{p}(\psi))
= \ & \eta_{\frak{so}_{2n+1}, \frak{sp}_{2n}}([(2b+1)^{a}(1)^{2m+1-a}])\\
= \ & ([(2b+1)^a(1)^{2m-a}]_{\Sp_{2n}})^t\\
= \ & [(2b+1)^{a-1}(2b)(2)(1)^{2m-a-1}]^t\\
= \ & [(a-1)^{2b+1}] + [(1)^{2b}] + [(1)^2] + [(2m-1-a)]\\
= \ & [(2m)(a+1)(a)^{2b-2}(a-1)].
\end{align*}

The strategy to prove Theorem \ref{main1} is as follows.
Given a symplectic partition $\ul{p}$ of $2n$ (that is, odd parts occur with even multiplicities),
Denote by $\ul{p}^{\Sp_{2n}}$ the $\Sp_{2n}$-expansion of $\ul{p}$, which is the smallest special symplectic partition that is
bigger than $\ul{p}$.
In \cite{JL15}, we proved the following theorem which provides a crucial reduction in the proof of Theorem \ref{main1}. The main idea of the proof will be explained in Section 5.3.

\begin{thm}[Theorem 4.1 \cite{JL15}]\label{special}
Let $\pi$ be an irreducible automorphic representation of
$\Sp_{2n}(\BA)$. If $\pi$ has a nonzero Fourier coefficient attached to a
non-special symplectic partition $\ul{p}$ of $2n$, then $\pi$ must have a nonzero
Fourier coefficient attached to $\ul{p}^{\Sp_{2n}}$, the $\Sp_{2n}$-expansion of the partition $\ul{p}$.
\end{thm}

If $a \leq 2m$ and $a$ is odd, then by the recipe of
obtaining $\Sp_{2n}$-expansion of a symplectic partition $\ul{p}$
given in Lemma 6.3.9 of \cite{CM93}, we can see that
$$[(2m)(a+1)(a)^{2b-2}(a-1)] = [(2m)(a)^{2b}]^{\Sp_{2n}}.$$
Hence it suffices to prove the following theorem.

\begin{thm}[\cite{JL14}]\label{main12}
For any global Arthur parameter of the form
$$
\psi=(\tau, 2b+1) \boxplus \boxplus_{i=2}^r (\tau_i,1)
$$
with $b \geq 1$, and $\tau \ncong \tau_i$ for any $2 \leq i \leq r$, the followings hold:
\begin{enumerate}
\item If $a=2m+1$, then $\CE_{\Delta(\tau, b)\otimes \sigma}$ has a nonzero Fourier coefficient attached to $[(a)^{2b}(2m)]$.
\item If $a \leq 2m$, then $\CE_{\Delta(\tau, b)\otimes \sigma}$ has a nonzero Fourier coefficient attached to $[(2m)(a)^{2b}]$.
\end{enumerate}
\end{thm}

{\textbf{Case \Rmnum{2}}:}\
$\psi$ has the following form
\begin{equation}\label{case2psi}
\psi=(\tau, 2b+1) \boxplus (\tau, 1) \boxplus \boxplus_{i=3}^r (\tau_i,1),
\end{equation}
where $b \geq 1$, $\tau \ncong \tau_i$, for any $3 \leq i \leq r$. Assume $\tau$ is an irreducible cuspidal representation of $\GL_a(\BA)$ with central character $\omega_{\tau}$, $\tau_i$ is an irreducible cuspidal representation of $\GL_{a_i}(\BA)$ with central character $\omega_{\tau_i}$, $3 \leq i \leq r$.
Then, by the definition of Arthur parameters, $\tau$ and $\tau_i$'s are all of orthogonal type; $\tau_i \ncong \tau_j$, for $3 \leq i \neq j \leq r$; $2n+1=a(2b+1)+ a + \sum_{i=3}^r {a_i}$, and $\omega_{\tau}^{2b+1} \cdot \omega_{\tau} \cdot \prod_{i=3}^r \omega_{\tau_i} = 1$. Consider the isobaric sum representation
$$\pi=\tau_3 \boxplus \cdots \boxplus \tau_r$$
of $\GL_{2m+1}(\BA)$, where $2m+1=\sum_{i=3}^r a_i = 2n+1-a(2b+2)$. Note that $\pi$ has central character
$\omega_{\pi}=\prod_{i=3}^r \omega_{\tau_i} = 1$.

By the theory of automorphic descent \cite[Theorem 3.1]{GRS11}, there is an irreducible generic cuspidal representation $\sigma$ of $\Sp_{2m}(\BA)$ such that $\sigma$ has the functorial transfer $\pi$. One can see that $L(s, \tau \times \sigma)$ is holomorphic at $s=1$ in this case.
The situation is similar to {\bf Case I}, and we refer to \cite[Section 2]{JL14} for more detailed discussion.

{\textbf{Case \Rmnum{3}}:}
$\psi$ has the following form
\begin{equation}\label{case3psi}
\psi=(\tau, 2b) \boxplus \boxplus_{i=2}^r (\tau_i,1),
\end{equation}
where $b \geq 1$. Note that in this case, $\tau$ is of symplectic type (and hence $a=2k$ is even),
while $\tau_i$ for all $2 \leq i \leq r$ are of orthogonal type.
Assume that $\tau$ is an irreducible cuspidal representation of $\GL_a(\BA)$ with central character $\omega_{\tau}$, and
$\tau_i$ is an irreducible cuspidal representation of $\GL_{a_i}(\BA)$ with central character $\omega_{\tau_i}$, $2 \leq i \leq r$.
Then, by the definition of Arthur parameters, one has that $2n+1=2ab + \sum_{i=2}^r {a_i}$, and $\prod_{i=2}^r \omega_{\tau_i} = 1$. Consider the isobaric sum representation
$$\pi=\tau_2 \boxplus \cdots \boxplus \tau_r$$
of $\GL_{2m+1}(\BA)$, where $2m+1=\sum_{i=2}^r a_i = 2n+1-2ab$. Note that $\pi$ has central character
$\omega_{\pi}=\prod_{i=2}^r \omega_{\tau_i} = 1$.

By the theory of automorphic descent \cite[Theorem 3.1]{GRS11}, there is an irreducible generic cuspidal representation $\sigma$ of $\Sp_{2m}(\BA)$ such that $\sigma$ has the functorial transfer $\pi$. Similarly, we have a residual Eisenstein series
$$
E(,\phi,s)(g)=E(g,\phi_{\Delta(\tau,b) \otimes \sigma},s)
$$
associated to any automorphic form
$$
\phi \in \CA(N_{ab}(\BA)M_{ab}(F) \bs \Sp_{2ab+2m}(\BA))_{\Delta(\tau,b) \otimes \sigma},
$$
following \cite{L76} and \cite{MW95}. By \cite{JLZ13}, this Eisenstein series
may have a simple pole at $\frac{b}{2}$, which is the right-most one.
Denote the representation generated by these residues at $s=\frac{b}{2}$
by $\CE_{\Delta(\tau, b)\otimes \sigma}$.
This residual representation is square-integrable. We note that if $L(\frac{1}{2}, \tau \times \sigma) \neq 0$, then the residual
representation $\CE_{\tau\otimes \sigma}$ is nonzero, and hence by the induction argument in \cite{JLZ13}, the
residual representation $\CE_{\Delta(\tau, b)\otimes \sigma}$ is also nonzero.
Finally, according to Section 6.2 of \cite{JLZ13}, the global Arthur parameter of $\CE_{\Delta(\tau,b)\otimes \sigma}$ is exactly
$$
\psi=(\tau, 2b)\boxplus_{i=2}^r (\tau_i,1)
$$
as in \eqref{case3psi}.

\begin{thm}[\cite{JL14}]\label{main3}
Assume that $a=2k$ and $L(\frac{1}{2}, \tau \times \sigma) \neq 0$.
If the residual representation $\CE_{\tau\otimes \sigma}$ of $\Sp_{4k+2m}(\BA)$, with $\sigma \ncong 1_{\Sp_0(\BA)}$,
has a nonzero Fourier coefficient attached to the partition $[(2k+2m)(2k)]$, then, for any $b \geq 1$, the residual
representation $\CE_{\Delta(\tau, b)\otimes \sigma}$ has a nonzero Fourier coefficient attached to the partition
$[(2k+2m)(2k)^{2b-1}]$.
\end{thm}

Note that if $\sigma \cong 1_{\Sp_0(\BA)}$, the assumption that
$L(\frac{1}{2}, \tau \times \sigma) \neq 0$ becomes $L(\frac{1}{2}, \tau) \neq 0$. In this case, it is proved in \cite[Theorem 2.1]{L13} that
$\mathfrak{p}^m(\CE_{\Delta(\tau, b)}) = \{[(2k)^{2b}]\}$, which confirms Part (3) of Conjecture \ref{cubmfc} and Conjecture \ref{singleton}
for the global Arthur parameter $\psi=(\tau,b)\boxplus(1,1)$.

We remark that the assumption that the residual representation $\CE_{\tau\otimes \sigma}$ of $\Sp_{4k+2m}(\BA)$,
with $\sigma \ncong 1_{\Sp_0(\BA)}$, has a nonzero Fourier coefficient attached to the partition $[(2k+2m)(2k)]$
is exactly Conjecture 6.1 in \cite{GJR04}, and hence Theorem \ref{main3} has a close connection to the Gan-Gross-Prasad conjecture (\cite{GGP12}),
as remarked in \cite{JL14}.

In this case, one has that $\ul{p}(\psi)=[(2b)^{a}(1)^{2m+1}]$, and
\begin{align*}
\eta_{\frak{so}_{2n+1}, \frak{sp}_{2n}}(\ul{p}(\psi))
= \ & \eta_{\frak{so}_{2n+1}(\BC), \frak{sp}_{2n}(\BC)}([(2b)^{a}(1)^{2m+1}])\\
= \ & [(2b)^a(1)^{2m}]^t\\
= \ & [(a)^{2b}] + [(2m)]\\
= \ & [(a+2m)(a)^{2b-1}],
\end{align*}
where $a=2k$ is even. When $L(\frac{1}{2}, \tau \times \sigma)$ is zero for
the Arthur parameter $\psi$ in \eqref{case3psi}, the corresponding automorphic $L^2$-packet $\wt{\Pi}_\psi(\epsilon_\psi)$
are expected to contain all cuspidal automorphic representations if it is not empty.
This is the case to be discussed in the next section.

\subsection{Part (3) of Conjecture \ref{cubmfc}: cuspidal case}
We briefly discuss the situation when the automorphic $L^2$-packet $\wt{\Pi}_\psi(\epsilon_\psi)$ contains all cuspidal members, if nonempty, for
non-generic global Arthur parameters $\psi\in\wt{\Psi}_2(G)$.

In this case, we may have to use the explicit construction of endoscopy
correspondences outlined in \cite{J14} in order to construct cuspidal members in the automorphic $L^2$-packet $\wt{\Pi}_\psi(\epsilon_\psi)$
for non-generic global Arthur parameters $\psi\in\wt{\Psi}_2(G)$. With such an explicit expression of cuspidal automorphic representations
in terms of their endoscopic data, we should calculate the Fourier coefficients attached to the partition
$\eta_{{\frak{g}^\vee,\frak{g}}}(\underline{p}(\psi))$. The technical key is how to express the global nonvanishing of such Fourier
coefficients by means of nonvanishing of certain data explicitly given by the global Arthur parameter $\psi$. This is one of the
on-going projects of the authors and will be discussed with more details in our future work.

\section{Other Topics Related to Fourier Coefficients}

As remarked before, it is important to figure $\frak{p}^m(\pi)$ for each member $\pi$ belonging to the automorphic
$L^2$-packet $\wt{\Pi}_\psi(\epsilon_\psi)$ for all global Arthur parameters $\psi\in\wt{\Psi}_2(G)$.
In general, the theory of Howe and Li on singular automorphic forms (\cite{H81} and \cite{Li92}) provides a lower bound for
$\frak{p}^m(\pi)$, since all cuspidal automorphic representations are non-singular. We may reformulate this in terms of
the notion of Fourier coefficients attached to nilpotent orbits or partitions. Meanwhile, we also intend to understand the properties of
the nilpotent orbits or the partitions in $\frak{p}^m(\pi)$ for general cuspidal automorphic representations $\pi$ of quasisplit
classical groups $G$.

\subsection{Problems on small cuspidal automorphic representations}
It is a classical problem in the theory of automorphic forms to understand singular automorphic forms. In \cite{H81}
Howe characterizes the singular automorphic forms of $\Sp_{2n}$ in terms of the theory of theta correspondences, based on his notion of
rank of unitary representations. In \cite{Li92}, it is proved that all cuspidal automorphic representations are non-singular
for all classical groups. Hence for cuspidal automorphic representations, it is interesting to understand those with top Fourier coefficients attached to small nilpotent orbits. Such representations are called  {\sl small} cuspidal automorphic representations. 

In the spirit of the explicit constructions of endoscopy correspondences in \cite{J14}, we only consider $F$-split classical groups $G$ here
and leave the more general case for future consideration. We are going to reformulate the result of Li in \cite{Li92} in terms of the
notion of Fourier coefficients of automorphic forms attached to partitions or nilpotent orbits. Following the line of ideas in \cite{H81} and
\cite{Li89}, we will investigate the relation between the structure of the set $\frak{p}^m(\pi)$ and the property of $\pi\in\CA_\cusp(G)$ 
in terms of the endoscopic classification of the discrete spectrum of Arthur.

First of all, we reformulate the full rank Fourier coefficients of automorphic forms in terms of partitions or nilpotent orbits.
We denote by $\udl{p}_\ns$ the partition corresponding to the non-singular Fourier coefficients.

{\bf Case ($G=\Sp_{2n}$):}\
The corresponding partition is $\udl{p}_\ns=[2^n]$. This is a special partition for $\Sp_{2n}$.

{\bf Case ($G=\SO_{2n+1}$):}\
The corresponding partition $\udl{p}_\ns$ is given as follows
$$
\udl{p}_\ns
=
\begin{cases}
[2^{2e}1]&\ \text{if}\ n=2e;\\
[2^{2e}1^3]&\ \text{if}\ n=2e+1.
\end{cases}
$$

{\bf Case ($G=\SO_{2n}$):}\
The corresponding partition $\udl{p}_\ns$ is given as follows
$$
\udl{p}_\ns
=
\begin{cases}
[2^{2e}]&\ \text{if}\ n=2e;\\
[2^{2e}1^2]&\ \text{if}\ n=2e+1.
\end{cases}
$$

Following \cite{JLS14}, which will be briefly discussed below, we know that for any automorphic representation $\pi$, the set
$\frak{p}^m(\pi)$ contains only special partitions. It is easy to check that when $G=\Sp_{2n}$ and $G=\SO_{2n}$, the
non-singular partitions $\udl{p}_\ns$ are special, but when $G=\SO_{2n+1}$, the non-singular partition $\udl{p}_\ns$ is not special.
Hence when $G=\SO_{2n+1}$, the partitions contained in $\frak{p}^m(\pi)$ as $\pi$ runs in $\CA_\cusp(G)$ should be bigger than or
equal to the following partition
$$
\udl{p}_\ns^G=
\begin{cases}
[32^{2e-2}1^2]&\ \text{if}\ n=2e;\\
[32^{2e-2}1^4]&\ \text{if}\ n=2e+1.
\end{cases}
$$
Following \cite{CM93}, $\udl{p}_\ns^G$ denotes the $G$-expansion of the partition $\udl{p}_\ns$.
Of course, when $G=\Sp_{2n}$ and $G=\SO_{2n}$, $\udl{p}_\ns^G=\udl{p}_\ns$.

\begin{prop}
The $G$-expansion of the non-singular partition, $\udl{p}_\ns^G$,
is a lower bound for partitions in the set $\frak{p}^m(\pi)$ as $\pi$ runs in $\CA_\cusp(G)$.
\end{prop}

According to known examples, one expects more precise structure for the set $\frak{p}^m(\pi)$ as $\pi$ runs in $\CA_\cusp(G)$. For instance,
\cite{GRS03} suggested that the set $\frak{p}^m(\pi)$ contains partitions with only even parts for $\Sp_{2n}$ or its metaplectic double cover;
and it is conjectured in \cite{G06} that
for split $\SO_{2n}$, the set $\frak{p}^m(\pi)$ contains partitions with odd parts with certain constraints. We expect that the same happens
to other classical groups.
On one hand, it is important to determine the set $\frak{p}^m(\pi)$ as $\pi$ runs in $\CA_\cusp(G)$. On the other hand, it is interesting to
construct explicitly cuspidal automorphic representations $\pi$ such that $\udl{p}_\ns^G\in\frak{p}^m(\pi)$. We will get back to this topic
in our future work.

\subsection{Maximal Fourier coefficients: stabilizer}
In \cite{GRS03}, for any irreducible cuspidal automorphic representation $\pi$ of $\Sp_{2n}(\BA)$ or $\wt{\Sp}_{2n}(\BA)$,
Ginzburg, Rallis and Soudry constructed an even partition $\ul{p}(\pi)$, i.e., consisting of only even parts,
which satisfies a property of ``maximal at every stage" (for the details of this property, see the proof of \cite[Theorem 2.7]{GRS03} or
\cite[Remark 5.1]{JL15}).
Write $\ul{p}(\pi) = [(2n_1)^{s_1} (2n_2)^{s_2} \cdots (2n_k)^{s_k} ]$ with $n_1 > n_2 > \cdots > n_k$ and $s_i \geq 1$.
Then, the Levi part of the stabilizer of any character $\psi_{\ul{p}, \CO}$ is isomorphic to $O_1 \times O_2 \times \cdots \times O_k$,
where $O_i$ is the orthogonal group with respect to the quadratic form $Q_i$, and the nilpotent orbit $\CO$ is parametrized
by the pair $(\ul{p}, \{Q_i\}_{i=1}^k)$ (\cite{W01}).

In \cite{JL15}, we show that all $O_i$'s are $F$-anisotropic. This confirms a conjecture of Ginzburg (\cite[Conjecture 4.3]{G06}).
As an easy corollary, if $F$ is a totally imaginary number field, then the equality $s_i \leq 4$ holds for any $i=1,\ldots, k$.
The local analogues of this result have been proved by the work of M{\oe}glin and Waldspureger (\cite{MW87}) and
of M{\oe}glin (\cite{M96}) for $p$-adic local fields, and by the work of B. Harris (\cite{H11}) for archimedean local fields.

If we assume that set $\frak{p}^m(\pi)$ contains only one partition, the above provides an understanding of the $F$-rational structure of
the nilpotent orbits with which the corresponding partition belonging to $\frak{p}^m(\pi)$. It is an ongoing project to extend this to
other classical groups.

\subsection{Maximal Fourier coefficients: special orbit}
In \cite{JLS14}, we extend the result in Theorem \ref{special} to all local ($p$-adic fields) and global representations of groups $\mathrm{G}=\Sp_{2n}$ and its double cover, $\SO_{2n}$, $\SO_{2n+1}$.

Note that in the $p$-adic field case, given a nilpotent orbit $\CO$ corresponding to a partition $\ul{p}$,
one can also define Jacquet modules with respect to the group $V_X$ and the character $\psi_X$ for $X \in \CO$.
As before,
we say $\ul{p}$ is a $\mathrm{G}$-partition if it is a symplectic partition of $2n$ when $\mathrm{G}=\Sp_{2n}$ and its double cover,
and it is an orthogonal partition of $N$ when $\mathrm{G}=\SO_N$. For the double cover of symplectic groups,
a symplectic partition of $2n$ is called {\sl metaplectic-special} if the number of even parts bigger than every odd part appearing
in the corresponding partition is odd (metaplectic-special partitions are first defined in \cite{M96} and are called anti-special there).
We shall say that $\ul{p}$ is {\sl special}, or {\sl $G$-special}, if it is symplectic, metaplectic
or orthogonal special, respectively. Given any symplectic partition $\ul{p}$ of $2n$, its metaplectic special
expansion $\ul{p}^{\wt{\Sp}_{2n}}$ is defined to be the smallest metaplectic special partition of $2n$ which is bigger or equal than $\ul{p}$.

\begin{thm}[\cite{JLS14}]\label{special2}
With notation as above, the following hold:
\begin{enumerate}
\item Let $k$ be a $p$-adic local field. Let $\pi$ be a smooth representation of $\mathrm{G}(k)$,
which is genuine when $\mathrm{G}(k)=\wt{\Sp}_{2n}(k)$. If $\pi$ has a nonzero Jacquet module attached to a
non-special $\mathrm{G}$-partition $\ul{p}$, then $\pi$ must have a nonzero
Jacquet module attached to $\ul{p}^{\mathrm{G}}$, the $\mathrm{G}$-expansion of the partition $\ul{p}$.
\item Let $F$ be a number field.
Let $\pi$ be an automorphic representation of
$\mathrm{G}(\BA)$, which is genuine when $\mathrm{G}(\BA_F)=\wt{\Sp}_{2n}(\BA_F)$.
If $\pi$ has a nonzero Fourier coefficient attached to a
non-special $\mathrm{G}$-partition $\ul{p}$, then $\pi$ must have a nonzero
Fourier coefficient attached to $\ul{p}^{\mathrm{G}}$, the $\mathrm{G}$-expansion of the partition $\ul{p}$.
\end{enumerate}
\end{thm}

As an easy corollary, we obtain the following result.

\begin{thm}\label{special3}
Let $F$ be a number field.
Let $\pi$ be an irreducible automorphic representation of
$\mathrm{G}(\BA)$, which is genuine when $\mathrm{G}(\BA_F)=\wt{\Sp}_{2n}(\BA_F)$.
Then any partition in the set $\frak{p}^m(\pi)$ is special. The same holds accordingly for $p$-adic local fields.
\end{thm}

The local part of Theorem \ref{special3} has been proved by M{\oe}glin (\cite{M96}) for classical groups, using the result of M{\oe}glin and Waldspurger (\cite{MW87}) on local character expansion.
The proof of the global part of Theorem \ref{special3} has been sketched by Ginzburg, Rallis and Soudry
(\cite{GRS03}) for symplectic groups, and briefly discussed by Ginzburg (\cite{G06}) for all split classical groups.

In \cite{JLS14}, we actually prove more. For the $p$-adic local field case, take $\pi$ to be a smooth representation of $\mathrm{G}(k)$,
which is genuine when $\mathrm{G}(k)=\wt{\Sp}_{2n}(k)$. Assume that $\pi$ has a nonzero Jacquet module attached to a
$\mathrm{G}$-partition $\ul{p}$, which can be viewed as a representation of the Levi part of the stabilizer of the character $\psi_{\ul{p}, \CO}$ (a product of certain symplectic and orthogonal groups).
If this Jacquet module has a nonzero Jacquet module attached to a minimal non-trivial partition
(in either a symplectic factor group or an orthogonal factor group),
then $\pi$ has a nonzero Jacquet module attached to a partition
that is bigger than $\ul{p}$ (for detailed description of this partition, see \cite{JLS14}).
If $\ul{p}$ is not special (as in Theorem \ref{special2}(1)), then one can see that this Jacquet module indeed has a nonzero Jacquet module attached to the minimal partition of a symplectic factor group. Hence, $\pi$ must have a nonzero Jacquet module attached to a partition (the $\rm G$-expansion of $\ul{p}$) that is bigger than $\ul{p}$.

For the global case, take $\pi$ to be an automorphic representation of $\mathrm{G}(\BA)$,
which is genuine when $\mathrm{G}(\BA)=\wt{\Sp}_{2n}(\BA)$. Assume that $\pi$ has a nonzero Fourier coefficient attached to a
$\mathrm{G}$-partition $\ul{p}$, which can be viewed as a function on the Levi part of the stabilizer of the character $\psi_{\ul{p}, \CO}$ (a product of certain symplectic and orthogonal groups).
If this Fourier coefficient has a nonzero Fourier coefficient attached to a minimal non-trivial partition
(in either a symplectic factor group or an orthogonal factor group),
then $\pi$ has a nonzero Fourier coefficient attached to a partition
that is bigger than $\ul{p}$ (for detailed description of this partition, see \cite{JLS14}).
If $\ul{p}$ is not special (as in Theorem \ref{special2}(2)), then one can see that this Fourier coefficient indeed has a nonzero Fourier coefficient attached to the minimal partition of a symplectic factor group. Hence, $\pi$ must have a nonzero Fourier coefficient attached to a partition (the $\rm G$-expansion of $\ul{p}$) that is bigger than $\ul{p}$. This is also the main idea of the proof of Theorem \ref{special}.
These results actually give an algorithm of raising nilpotent orbits which provide non-zero local Jacquet modules or global Fourier
coefficients for a representation (local or global automorphic). Their potential applications will be considered in our future work.

\section{Roots exchange and Fourier coefficients}

In the study of Fourier coefficients of automorphic forms, in particular concerning the global nonvanishing property, a technical lemma has
been very useful in the theory. We try to extend it to a few versions, which will be more convenient for future applications.

Let $H$ be any $F$-quasisplit classical group, including the general linear group.
First, we recall Lemma 5.2 of \cite{JL13a}, which is also formulated in a slight variation as in Corollary 7.1 of \cite{GRS11}.
Note that the proof of Lemma 5.2 of \cite{JL13a} is valid for $H(\BA)$.

Let $C$ be an $F$-subgroup of a maximal unipotent subgroup of $H$, and let $\psi_C$ be a non-trivial character of $[C] = C(F) \bs C(\BA)$.
$X, Y$ are two unipotent $F$-subgroups, satisfying the following conditions:
\begin{itemize}
\item[(1)] $X$ and $Y$ normalize $C$;
\item[(2)] $X \cap C$ and $Y \cap C$ are normal in $X$ and $Y$, respectively, $(X \cap C) \bs X$ and $(Y \cap C) \bs Y$ are abelian;
\item[(3)] $X(\BA)$ and $Y(\BA)$ preserve $\psi_C$;
\item[(4)] $\psi_C$ is trivial on $(X \cap C)(\BA)$ and $(Y \cap C)(\BA)$;
\item[(5)] $[X, Y] \subset C$;
\item[(6)]  there is a non-degenerate pairing $(X \cap C)(\BA) \times (Y \cap C)(\BA) \rightarrow \BC^*$, given by $(x,y) \mapsto \psi_C([x,y])$, which is
multiplicative in each coordinate, and identifies $(Y \cap C)(F) \bs Y(F)$ with the dual of
$
X(F)(X \cap C)(\BA) \bs X(\BA),
$
and
$(X \cap C)(F) \bs X(F)$ with the dual of
$
Y(F)(Y \cap C)(\BA) \bs Y(\BA).
$
\end{itemize}

Let $B =CY$ and $D=CX$, and extend $\psi_C$ trivially to characters of $[B]=B(F)\bs B(\BA)$ and $[D]=D(F)\bs D(\BA)$,
which will be denoted by $\psi_B$ and $\psi_D$ respectively.

\begin{lem}[Lemma 5.2 of \cite{JL13a}]\label{nvequ}
Assume that $(C, \psi_C, X, Y)$ satisfies all the above conditions. Let $f$ be an automorphic form on $H(\BA)$. Then
$$\int_{[C]} f(cg) \psi_C^{-1}(c) dc \equiv 0, \forall g \in H(\BA),$$
if and only if
$$\int_{[D]} f(ug) \psi_D^{-1}(u) du \equiv 0, \forall g \in H(\BA),$$
if and only if
$$\int_{[B]} f(vg) \psi_B^{-1}(v) dv \equiv 0, \forall g \in H(\BA).$$
\end{lem}

It is well known that the Fourier transformation on a locally compact abelian group can be reformulated in terms of Heisenberg groups and
the associated Weil representations as explained in \cite{We64}. Lemma \ref{nvequ} can be viewed as an extension of this basic fact to
such a general set-up, where a Heisenberg group plays an essential role. The result of Lemma \ref{nvequ} is also an extension to such
a generality of the basic fact about Fourier-Jacobi coefficients for automorphic forms as formulated in \cite{GRS03}. As explained
in \cite{GRS11}, this lemma played essential roles when one tries to understand the non-vanishing of certain Fourier coefficients in the
theory of automorphic descents. When we try to understand Fourier coefficients related to Conjecture \ref{cubmfc}, we face more general setting
than what can simply be covered by Lemma \ref{nvequ}, and hence we try to extend Lemma \ref{nvequ} to a few different versions, which
will be useful in the theory of Fourier coefficients. More discussions from representation-theoretic point of view can be found in
\cite{JLS14}.

For simplicity, we always use $\psi_C$ to denote its extensions $\psi_B$ and $\psi_D$ when we use Lemma \ref{nvequ}.
The following lemma is an extension of Lemma \ref{nvequ} via induction, proof of which will be omitted.

\begin{lem}\label{nvequ2}
Assume the quadruple $(C, \psi_C, X, Y)$ satisfies the following conditions:
$X = \{X_i\}_{i=0}^r$, $Y = \{Y_i\}_{i=1}^{r+1}$, $X_0 = Y_{r+1}= \{1_{H(\BA)}\}$, for $1 \leq i \leq r$, each quadruple
$$(X_{i-1} \cdots X_1 C Y_r \cdots Y_{i+1}, \psi_C, X_i, Y_i)$$
satisfies all the conditions for Lemma \ref{nvequ}.
Let $f$ be an automorphic form on $H(\BA)$.
Then
$$\int_{[X_{r} \cdots X_1 C]} f(xcg) \psi_C^{-1}(c) dcdx \equiv 0, \forall g \in H(\BA),$$
if and only if
$$\int_{[CY_r \cdots Y_{1}]} f(cyg) \psi_C^{-1}(c) dydc \equiv 0, \forall g \in H(\BA).$$
\end{lem}

The following lemma follows from Corollary 7.2 of \cite{GRS11},
and is proved in the proof of Theorem 2.1 of \cite{GJS12} (see
formula (2.26) and its proof therein), the argument also appears in the proof of \cite[Lemma 1, p. 895]{GRS99}.

\begin{lem}[\cite{GJS12} (2.26), \cite{GRS99} Lemma 1, p. 895]\label{nvequ3}
Assume the quadruple $(C, \psi_C, X, Y)$ is as in Lemma \ref{nvequ}.
Let $\pi$ be an irreducible automorphic representation of $H(\BA)$, and let $f \in \pi$. Then
there is an $f' \in \pi$, such that
$$\int_{[B]} f(vg) \psi_B^{-1}(v) dv
= \int_{[D]} f'(ug) \psi_D^{-1}(u) du, \forall g \in H(\BA).$$
\end{lem}

\begin{proof}
We sketch the proof as follows.
Without loss of generality (via replacing $f$ by its right $g$ translation), we only have to prove
$$\int_{[B]} f(v) \psi_B^{-1}(v) dv
= \int_{[D]} f'(u) \psi_D^{-1}(u) du.$$
By Lemma 7.1 of \cite{GRS11},
$$\int_{[B]} f(v) \psi_B^{-1}(v) dv
=\int_{(Y\cap C)(\BA)\bs Y(\BA)} \int_{[D]} f(uy) \psi_D^{-1}(u) dudy.$$
By Corollary 7.2 of \cite{GRS11},
there exist $f_1, \ldots, f_r \in \pi$,
and Schwartz functions $\phi_1, \ldots, \phi_r \in \CS((Y\cap C)(\BA)\bs Y(\BA))$, such that for all $y \in (Y\cap C)(\BA)\bs Y(\BA)$,
$$\int_{[D]} f(uy) \psi_D^{-1}(u) du=
\sum_{i=1}^r \phi_i(y)\int_{[D]} f_i(uy) \psi_D^{-1}(u) du.$$
Therefore,
\begin{align*}
\int_{[B]} f(v) \psi_B^{-1}(v) dv
= & \int_{(Y\cap C)(\BA)\bs Y(\BA)} \int_{[D]} f(uy) \psi_D^{-1}(u) dudy\\
= & \int_{(Y\cap C)(\BA)\bs Y(\BA)}\sum_{i=1}^r \phi_i(y)\int_{[D]} f_i(uy) \psi_D^{-1}(u) du dy\\
= & \int_{[D]} \sum_{i=1}^r \int_{(Y\cap C)(\BA)\bs Y(\BA)} \phi_i(y)f_i(uy)dy \psi_D^{-1} du.
\end{align*}
If we take
$$
f'(u)=\sum_{i=1}^r \phi_i * f_i(u)
=\sum_{i=1}^r \int_{(Y\cap C)(\BA)\bs Y(\BA)} \phi_i(y)f_i(uy)dy,
$$
then $f'=\sum_{i=1}^r \phi_i * f_i \in \pi$ and we obtain
$$
\int_{[B]} f(v) \psi_B^{-1}(v) dv
= \int_{[D]} f'(u) \psi_D^{-1} du.
$$
\end{proof}

For simplicity, we always use $\psi_C$ to denote its extensions $\psi_B$ and $\psi_D$ when we use Lemma \ref{nvequ3}.
Following Lemma \ref{nvequ3}, we obtain the following result which is analogous to Lemma \ref{nvequ2}.

\begin{lem}\label{nvequ4}
Assume the quadruple $(C, \psi_C, X, Y)$ satisfies the following conditions:
$X = \{X_i\}_{i=0}^r$, $Y = \{Y_i\}_{i=1}^{r+1}$, $X_0 = Y_{r+1}= \{1_{H(\BA)}\}$, for $1 \leq i \leq r$, each quadruple
$$(X_{i-1} \cdots X_1 C Y_r \cdots Y_{i+1}, \psi_C, X_i, Y_i)$$
satisfies all the conditions for Lemma \ref{nvequ}.
Let $\pi$ be an irreducible automorphic representation of $H(\BA)$, and let $f \in \pi$.
Then there is an $f' \in \pi$, such that
$$\int_{[CY_r \cdots Y_{1}]} f(cyg) \psi_C^{-1}(c) dydc =\int_{[X_{r} \cdots X_1 C]} f'(xcg) \psi_C^{-1}(c) dcdx, \forall g \in H(\BA).$$
\end{lem}

\begin{proof}
The proof of this lemma is very similar to that of Lemma \ref{nvequ2}. We apply Lemma \ref{nvequ3} in this proof whenever we
apply Lemma \ref{nvequ} in the proof of Lemma \ref{nvequ2}. To be complete, we include full detail here.

Since $X = \{X_i\}_{i=1}^r$, $Y = \{Y_i\}_{i=1}^r$, and for $1 \leq i \leq r$ each quadruple
$$(X_{i-1} \cdots X_1 C Y_r \cdots Y_{i+1}, \psi_C, X_i, Y_i)$$
satisfies all the conditions for Lemma \ref{nvequ}, we first apply Lemma \ref{nvequ3} to the following
quadruple
$$(C Y_r \cdots Y_{2}, \psi_C, X_1, Y_1),$$
and obtain that there exists $f^{(1)} \in \pi$, such that
\begin{align}\label{sec2equ6}
\begin{split}
& \int_{[CY_r \cdots Y_{1}]} f(cyg) \psi_C^{-1}(c) dydc \\
= \ & \int_{[X_1 C Y_r \cdots Y_{2}]} f^{(1)}(xcyg) \psi_C^{-1}(u) dydcdx, \forall g \in H(\BA).
\end{split}
\end{align}
Next we apply Lemma \ref{nvequ3} to the following
quadruple
$$(X_1 C Y_r \cdots Y_{3}, \psi_C, X_2, Y_2),$$
and obtain that there exists $f^{(2)} \in \pi$, such that
\begin{align}\label{sec2equ7}
\begin{split}
& \int_{[X_1 C Y_r \cdots Y_{2}]} f^{(1)}(cyg) \psi_C^{-1}(c) dydc \\
= \ & \int_{[X_2 X_1 C Y_r \cdots Y_{3}]} f^{(2)}(xcyg) \psi_C^{-1}(u) dydcdx, \forall g \in H(\BA).
\end{split}
\end{align}
The same argument is apply to the following sequence of quadruples, with help of Lemma \ref{nvequ3},
\begin{align}\label{sec2equ8}
\begin{split}
& (X_{2} \cdots X_1 C Y_r \cdots Y_{4}, \psi_C, X_3, Y_3),\\
& (X_{3} \cdots X_1 C Y_r \cdots Y_{5}, \psi_C, X_4, Y_4),\\
& \cdots \\
& (X_{r-1} \cdots X_1 C, \psi_C, X_r, Y_r).
\end{split}
\end{align}
For example, after applying Lemma \ref{nvequ3} to the quadruple
$$(X_{i-1} \cdots X_1 C Y_r \cdots Y_{i+1}, \psi_C, X_i, Y_i),$$
we get that there exists $f^{(i)} \in \pi$, such that
\begin{align}\label{sec2equ9}
\begin{split}
& \int_{[X_{i-1} \cdots X_1 C Y_r \cdots Y_{i}]} f^{(i-1)}(cyg) \psi_C^{-1}(c) dydc \\
= \ & \int_{[X_{i} \cdots X_1 C Y_r \cdots Y_{i+1}]} f^{(i)}(xcyg) \psi_C^{-1}(u) dydcdx, \forall g \in H(\BA).
\end{split}
\end{align}
After applying Lemma \ref{nvequ3} to the last quadruple
$$(X_{r-1} \cdots X_1 C, \psi_C, X_r, Y_r),$$
we get that there exists $f^{(r)} \in \pi$, such that
\begin{align}\label{sec2equ10}
\begin{split}
& \int_{[X_{r-1} \cdots X_1 C Y_{r}]} f^{(r-1)}(cyg) \psi_C^{-1}(c) dydc \\
= \ & \int_{[X_{r} \cdots X_1 C]} f^{(r)}(xcyg) \psi_C^{-1}(u) dydcdx, \forall g \in H(\BA).
\end{split}
\end{align}
Let $f'=f^{(r)} \in \pi$. From \eqref{sec2equ6} - \eqref{sec2equ10}, we can conclude that
$$\int_{[CY_r \cdots Y_{1}]} f(cyg) \psi_C^{-1}(c) dydc =\int_{[X_{r} \cdots X_1 C]} f'(xcg) \psi_C^{-1}(c) dcdx, \forall g \in H(\BA).$$
This completes the proof of the lemma.
\end{proof}

\end{document}